\documentclass{amsart}

\usepackage[margin=1.2in]{geometry}

\usepackage{amssymb}
\usepackage{amsmath}
\allowdisplaybreaks

\usepackage{xcolor}
\usepackage[all]{xy}
\usepackage{hyperref}
\usepackage{amssymb}
\usepackage{amsthm}
\usepackage{tikz-cd}

\newcommand{\R}{\mathbb R}
\newcommand{\F}{\mathbb F}

\newcommand{\Q}{\mathbb Q}
\newcommand{\Z}{\mathbb Z}

\newcommand{\fp}{\mathfrak p}

\newcommand{\fQ}{\mathfrak Q}

\newcommand{\co}{\mathcal{O}}
\newcommand{\fq}{\mathfrak q}

\newcommand{\Gal}{\mathrm{Gal}}

\newcommand{\ord}{\mathrm{ord}}

\newcommand{\N}{\mathrm{N}}

\newcommand{\Cl}{\mathrm{Cl}}
\newcommand{\cl}{\mathrm{cl}}

\newcommand{\leg}[2]{\left(\frac{#1}{#2}\right)}

\numberwithin{equation}{section}

\theoremstyle{plain}
\newtheorem{thm}{Theorem}[section]

\newtheorem{prop}[thm]{Proposition}
\newtheorem{cor}[thm]{Corollary}
\newtheorem{lem}[thm]{Lemma}

\theoremstyle{definition}

\newtheorem{rmk}[thm]{Remark}

\newtheorem*{Milovic's conjecture}{Milovic's conjecture}
\newtheorem*{Equivalent Form of the Conjecture}{Equivalent Form of the Conjecture}
\newtheorem*{conj*}{Conjecture}

\begin{document}
\title{On Class Numbers of Pure Quartic fields}

\author{Jianing Li}
\address{Wu Wen-Tsun Key Laboratory of Mathematics,  School of Mathematical Sciences, University of Science and Technology of China, Hefei, Anhui 230026, China}
\email{lijn@ustc.edu.cn}

\author{Yue Xu}
\address{Wu Wen-Tsun Key Laboratory of Mathematics,  School of Mathematical Sciences, University of Science and Technology of China, Hefei, Anhui 230026, China}
\email{wasx250@mail.ustc.edu.cn}



\maketitle

\begin{abstract}
Let $p$ be a prime. The  $2$-primary part of the class group of the pure quartic field $\mathbb{Q}(\sqrt[4]{p})$   has been determined  by Parry and Lemmermeyer when $p \not\equiv \pm 1\bmod 16$.  In this paper,  we improve the known results in  the case $p\equiv \pm 1\bmod 16$. In particular,   we determine all primes $p$ such that $4$ does not divide the class number of $\mathbb{Q}(\sqrt[4]{p})$. 
 We also  conjecture  a relation between the class numbers of  $\mathbb{Q}(\sqrt[4]{p})$ and $\mathbb{Q}(\sqrt{-2p})$.  We show that  this conjecture implies a  distribution  result of the $2$-class numbers of $\mathbb{Q}(\sqrt[4]{p})$.
\end{abstract}

\section{Introduction}
Let $p$ be a prime number.  Let $K$ be the pure quartic number field $\Q(\sqrt[4]{p})$.  The goal of this paper is to study the $2$-primary part of the class group $\Cl_K$ of $K$.   This question  has been studies by Parry  \cite{Parry} and Lemmermeyer \cite{Monsky}.      Parry showed that this group  is cyclic. So the question becomes to determine the exact divisibility of $2$-powers of the class number $h_K$ of $K$.    We list the known results.
\begin{enumerate}
	\item[(i)] If $p=2$ or $p\equiv \pm3  \bmod 8$,  then $2\nmid h_K$. 
	\item[(ii)]  If $p\equiv \pm 7\bmod 16$, then $2\parallel h_K$. The case $p\equiv -7\bmod 16$ is due to Lemmermeyer  \cite{Monsky}.
	\item [(iii)] If $p\equiv \pm 1\bmod 16$, then $2\mid h_K$.   In the case  $p\equiv 1\bmod 16$,   $2\parallel h_K$   if the quartic residue symbol $\leg{2}{p}_4$ equals to  $-1$.  
\end{enumerate}
We remark that in fact one  has $2\nmid h_{\Q(\sqrt[2^n]{p})}$   if $p\equiv \pm3 \bmod 8$ for any $n\geq 1$ and   $2\parallel h_{\Q(\sqrt[2^n]{p})}$ if $p\equiv 7\bmod 16$ for any $n\geq 2$, see \cite{ljn}.  Thus the remain problems are   to determine the $2$-divisibilities of $h_K$ when $p\equiv \pm 1\bmod 16$.   Our first result is the following.
\begin{thm}\label{thm: 1mod16}
	Let $p\equiv 1\bmod 16$ be a prime. Then  $2\parallel  h_K$   if and only if the quartic residue symbol $\leg{2}{p}_4=-1$. 
\end{thm}

 By Chebotarev's  density theorem,  the natural  density of the  set $\{p \text{ primes}:   p\equiv 1\bmod 16 \text{ and } 2\parallel h_K \}$ is $\frac{1}{16}$, see Remark~\ref{rmk: 1mod16}.

 We then turn to  study the case $p\equiv -1\bmod 16$.  We prove  the following improvement of Parry's and therefore we determine all primes $p$ such that $2\parallel h_K$.

\begin{thm}\label{thm: 4_divides pure}
Let $p\equiv -1 \bmod 16$ be a prime. Then  the $2$-primary part of $\Cl_K$ is cyclic and the class number $h_K$ is divisible by $4$.
\end{thm}

By class field theory, $K$ admits a unique unramified cyclic quartic extension  which we call it $4$-Hilbert class field of $K$. It is not hard to see that $K(\sqrt{2})$ is an unramified quadratic extension of $K$.
We construct the $4$-Hilbert class field in terms of the units of  the ring of integer in the quartic field $\Q(\sqrt{2\sqrt{p}})$. It can be shown  that there exists a totally positive unit $\xi$  such that the unit group $\co^\times_{\Q(\sqrt{2\sqrt{p}})}$ is generated by $\xi,\bar{\xi},-1$ and  the relative norm  $N_{\Q(\sqrt{2\sqrt{p}})/\Q(\sqrt{p})}(\xi)=\xi\bar{\xi}$ of $\xi$  is the fundamental unit of $\Q(\sqrt{p})$ (see Proposition~\ref{prop: units of K'}). 

\begin{thm}\label{thm: 4_hilbert class}
Let $p\equiv -1\bmod 16$ be a prime and $\xi$ as above.   Then the $4$-Hilbert class field of $K$ is $K(\sqrt{\xi})$.
\end{thm}

We observe that there  are some relations between $h_K$ and $h(-2p)$, where $h(-2p)$ is   class number of the  imaginary quadratic field $\Q(\sqrt{-2p})$. 
Due to  Gauss, we know that   the $2$-primary part of $\Cl_{\Q(\sqrt{-2p})}$ is cyclic and nontrivial. We list the known results on  $2$-divisibility of $h(-2p)$.  For  $p\equiv \pm 1\bmod 8$,  write $p=u^2-2v^2$ with   $u,v\in \mathbb{N}$ and $u\equiv 1\bmod 4$.    The following results are due to R\'edei \cite{Redei},  Reichardt \cite{Reichardt}, Hasse \cite{Hasse} and Leonard-Williams \cite{LW82}. We refer the readers to  \cite{LW82}.

\begin{enumerate}
\item [(i)] $2 \parallel   h(-2p)$ if and only if $p\equiv \pm 3\bmod 8$. 
\item [(ii)]  Suppose  $p\equiv  1\bmod 8$.   We have that  $4 \parallel h(-2p)$ if and only if $\leg{2}{p}_4=-1$ and that $8\parallel h(-2p)$ if and only if $\leg{u}{p}_4=-1$.  
\item [(iii)] Suppose $p\equiv -1\bmod 8$. We have that  $4\parallel  h(-2p)$ if and only if $p\equiv 7\bmod 16$ and that  $8\parallel h(-2p)$ if and only if $p\equiv -1\bmod 16$ and  $(-1)^{\frac{p+1}{16}}\leg{2u}{v}=-1$. 
\end{enumerate}
The  residue  symbols $\leg{u}{p}_4$  and  $\leg{2u}{v}$ are independent of the choices of $u,v$  (see \cite{LW82} and  Lemma~\ref{(p) is well-defined}).     By comparing the results on $h_K$,   one finds that if $p\equiv \pm 3\bmod 8$ or  $p\equiv 7\bmod 16$,  then   $\ord_2(h(-2p))=\ord_2(h_K)+1$. 
 Based on  numerical evidence, we propose the following conjecture.

\begin{conj*}\label{conj: 1}
	$(1)$	If $p\equiv 15\bmod 32$ is a prime, then  $4\parallel h_K \Longleftrightarrow  16\mid h(-2p)$.
	
	$(2)$ If $p\equiv 31 \bmod 32$ is a prime, then $ 4\parallel h_K \Longleftrightarrow  8 \parallel h(-2p)$.
\end{conj*}
	
By the above results on $h(-2p)$, this conjecture is  equivalent to the following. 
	
\begin{Equivalent Form of  the Conjecture}
 Let $p\equiv -1\bmod 16$ be a prime number.  Write $p=u^2-2v^2$ with $u,v\in \mathbb{N}$. Let $(p)$ denote the Jacobi symbol  $\leg{2u}{v}$.   Then
 \[4\parallel  h_K \Longleftrightarrow (p)=-1. \]
\end{Equivalent Form of the Conjecture}

Based on Cohen-Lenstra heuristic, Milovic \cite[Conjecture 1]{Milovic}   conjectures that for  each $k\geq 1$,  the natural density of  the set $\{p : p\equiv -1\bmod 4  \text{ and } 2^k\parallel h(-2p) \} $ is $\frac{1}{2^{k+1}}$. For $k=1, 2$, this  follows from the above results on $h(-2p)$ and the Dirichlet's density theorem  on arithmetic progressions. Milovic  \cite{Milovic} proves  the case $k=3$ by showing that  
\begin{equation}\tag{$\ast$}\label{eq: milovic}
\lim\limits_{X \rightarrow \infty} \frac{\#\{p \leq X: p \equiv -1 \bmod16 \text{ and } (-1)^{\frac{p+1}{16}}(p)=-1\}}{\#\{p \leq X: p \equiv -1 \bmod16 \}}=\frac{1}{2}.
\end{equation}
Use  his method, we  show the following.
\begin{thm}\label{thm:density} Let $(p)$ be as in the above conjecture. Then
\[\lim\limits_{X \rightarrow \infty} \frac{\#\{p \leq X, p \equiv -1 \bmod16: (p)=-1\}}{\#\{p \leq X, p \equiv -1 \bmod16 \}}=\frac{1}{2}.\]
\end{thm}

Thus  we  obtain the following result.
\begin{cor}
	Assume the above conjecture  holds. Then
	\[ \lim_{X\rightarrow \infty}\frac{\# \{p\leq X: p\equiv -1\bmod  16 \text{ and } 4\parallel  h_K \}}{\# \{p\leq X: p\equiv -1\bmod 16 \}}=\frac{1}{2}.\]
\end{cor}

There are $4927$ primes $p$ such that  $p<10^6$ and $p\equiv 15\bmod 32$. 
Pari-gp \cite{Pari} shows that   there are $2416$ primes $p$ such that $4\parallel h_{K}$ and there are $2511$  primes $p$ such that $8\parallel h_{K}$. Combine with \eqref{eq: milovic} and the above corollary,  one can   see  that  (under the above Conjecture)
\[\lim\limits_{X \rightarrow \infty} \frac{\#\{p \leq X: p \equiv 15 \bmod32 \text{ and }4\parallel h_K\}}{\#\{p \leq X: p \equiv 15 \bmod32 \}}=\lim\limits_{X \rightarrow \infty} \frac{\#\{p \leq X:  p \equiv 31 \bmod32 \text{ and } 4\parallel h_K\}}{\#\{p \leq X: p \equiv 31 \bmod32 \}}=\frac{1}{2}.\]
This partially explains the above data.   

The rest of this paper is organized as follows. In \S 2, we give some preliminaries and prove the result for $p\equiv 1\bmod 16$. In section 3, we prove Theorem~\ref{thm: 4_divides pure} and \ref{thm: 4_hilbert class}. In section 4, we discuss the density results on class numbers.

\subsection*{Acknowledgement}
Research is partially supported by Anhui Initiative in Quantum Information Technologies (Grant No. AHY-10200) and the Fundamental Research Funds for the Central Universities(no. WK0010000058). The authors thank  Franz Lemmermeyer for  helpful
 email exchanges.

\section{Preliminary and the case $p\equiv 1\bmod 16$}
In this section, we state Chevalley's ambiguous class number formula which is the main tool we used. Then we give the proof of the  cyclicity of the $2$-primary part of $\Cl_K$ for any $p$.   After that  we  prove that if $p\equiv 1\bmod 16$, then  $\leg{2}{p}_4=-1$  if and only if $2\parallel h_K$.

    Let $M/L$ be a cyclic  extension of number fields with Galois group $G$.   Then the ambiguous class number formula states as

\begin{equation}\label{eq: chevalley}
|\Cl^G_M|=|\Cl_L| \frac{\prod_{v}e_v}{[M:L]} \frac{1}{[\co^\times_L:\co^\times_L\cap N(M^\times)]}.
\end{equation}
Here $e_v$ is the ramification index of $v$ and the products run over all places of $L$ (including the infinite places). The norm $N$ is from $M$ to $L$.  See \cite{Lemmermeyer2} for a proof of this result.  If $G$ is a cyclic $\ell$-group where $\ell$ is a prime, then  $\ell\nmid |\Cl^G_M|$ implies  that  $\ell \nmid |\Cl_M|$. Because for $a\not\in \Cl^G_M$, the cardinality of the orbit of $a$ is divisible by $\ell$. Hence $|\Cl_M|\equiv |\Cl^G_M|\bmod \ell$.

\begin{prop}[Parry]\label{prop: cyclic}
Let $p$ be a prime. Then the $2$-primary part  of the class group of $K$ is cyclic.
\end{prop}

\begin{proof}
Let $A$  be the $2$-primary part of $\Cl_K$. Put $k=\Q(\sqrt{p})$ and   $G=\Gal(K/k)=\{1,\sigma\}$. It is well-known that the class number $h_k$  of $k$ is odd.  Thus we have $a^\sigma a=1$ for $a\in A$. This implies $A^G=A[2]:=\{a\in A|a^2=1\}$.   Applying Chevalley's formula on the quadratic extension $K/k$ gives
\[|A[2]|=|A^G|=\frac{\prod_{v} e_v}{2} \frac{1}{[\co^\times_{k}:\co^\times_{k}\cap N(K^\times)]}. \]
We have $-1\notin N(K^\times)$ since one of the  infinite places of $k$  is ramified. This implies $[\co^\times_{k}:\co^\times_{k}\cap N(K^\times)]=2 \text{ or } 4$.  We will show  that  for any odd prime $p$ one has $\prod_{v} e_v=8$  where $v$ runs over all places of $k$. It follows that  $|A[2]|=1 \text{ or } 2$. Hence  $A$ is trivial or cyclic.

 Obviously, $K/k$ is unramified outside places above $2,p$ and $\infty$. Here $\infty$ denotes the  real place  $\infty$  such that $\infty(\sqrt{p})<0$.   Also note that $(\sqrt{p})\co_k$ is ramified.  Thus $\prod_{v\nmid 2}e_v=4$.   We compute the ramification index at $2$ as follows.

If $p\equiv 3\bmod 4$, then $(x+1)^4-p$ is an Eisenstein polynomial in $\Q_2[x]$. Thus $2$ is totally ramified in $K/\Q$. This implies that $\prod_{v}e_v=8$ where $v$ runs over all places of $k$.

If $p\equiv 5 \bmod 8$, then $2$ is inert in $k$. Because $(x+1)^2-\sqrt{p}$ is an Eisenstein polynomial in $\Q_2(\sqrt{p})$, we have $2\co_{k}$ is ramified in $K$. Hence $\prod_{v}e_v=8$.

If $p\equiv 1 \bmod 16$, $x^4-p$ has solutions $\pm \sqrt[4]{p}$ in $\Q_2$.  Then $\sqrt{p}=(\sqrt[4]{p})^2\equiv 1 \bmod 8$.  Thus in $\Q_2[x]$,  we have  $x^4-p=(x-\sqrt[4]{p})(x+\sqrt[4]{p})(x^2+\sqrt{p})$. Note that  $(x+1)^2+\sqrt{p}$ is an Eisenstein polynomial. Therefore, there are three primes  in $K$ above $2$ and exactly one of them ramified.   This implies that $\prod_v{e_v}=8$.

If $p\equiv 9 \bmod 16$, then $x^4-p=(x^2-\sqrt{p})(x^2+\sqrt{p})$, and $\sqrt{p}\equiv \pm 3\bmod 8$. We have  $x^2\pm \sqrt{p}$ are irreducible in $\Q_2[x]$.  Because   $\Q_2(\sqrt{3})/\Q_2$ is ramified and $\Q_2(\sqrt{-3})/\Q_2$ is unramified.    Thus $\prod_v{e_v}=8$.  

 The proof of the proposition is complete.\end{proof}

Now we give the proof of Theorem~\ref{thm: 1mod16}. Much of the proof are as same as Lemmermeyer's proof of that  if  $p\equiv 9\bmod 16$ implies $2\parallel h_K$. The  only difference is that  in the case $p\equiv 1\bmod 16$, we need to investigate all the units of a quartic field. 

\begin{proof}[Proof of the Theorem~\ref{thm: 1mod16}]
 For  $p\equiv 1\bmod 8$,  the following are proved in \cite{Monsky}.     Let $F$ be the unique quartic subfield of the $p$-th cyclotomic field $\Q(\zeta_p)$.  Then $L:=FK$ is a quadratic unramified extension of $K$ and $F$ is the field $\Q(\sqrt{\epsilon \sqrt{p}})$ where $\epsilon$ is the fundamental unit of $k:=\Q(\sqrt{p})$.  One has 
\[2\parallel h_K \quad \text{ if and only if } \quad  2\nmid h(L).\] To prove $2\nmid h(L)$, Lemmermeyer uses Chevalley's formula on $L/k(\sqrt{\epsilon})$.  The class number of $k(\sqrt{\epsilon})$ is odd. This can be proved by the fact that $2\mid h_k$ and $\epsilon$ has norm $-1$.  The  primes of $k(\sqrt{\epsilon})$ ramified in $L$ are the two prime ideals above $p$.   In the case  $p\equiv 9\bmod 16$,   $\sqrt{\epsilon}$ is not a norm  of $L^\times$.  (In the case  $p\equiv 1\bmod 16$, $-1$ and $\sqrt{\epsilon}$  are  norms of $L^\times$.)  In particular the unit index 
\[[\co^\times_{k(\sqrt{\epsilon})}:\co^\times_{k(\sqrt{\epsilon})}\cap N(L^\times)]\geq 2.  \]
So one obtains $2\nmid h(L)$, hence $2\parallel h_K$ when $p\equiv 9\bmod 16$.

In order to prove our result, we need to show that for  $p\equiv 1\bmod 16$,  
\begin{equation}\label{eq: 1mod16}
[\co^\times_{k(\sqrt{\epsilon})}:\co^\times_{k(\sqrt{\epsilon})}\cap N(L^\times)]=2  \quad \text{ if and only if } \quad  \leg{2}{p}_4=-1. 
\end{equation}
 
 Since $p\equiv 1\bmod 16$, we write  $(2){\co_k}=\fq \bar{\fq}$.  Then we may assume $\fq$ is ramified in $K$ and $\bar{\fq}$ splits in $K$. Since $\fq,\bar{\fq}$ is unramified in $F$. We have $\fq$ is ramified in $k(\sqrt{\epsilon})$ and $\bar{\fq}$ is unramified in $k(\sqrt{\epsilon})$.    Then $\fq\co_{k(\sqrt{\epsilon})}=\fQ^2$. Let $\pi \in \co_k$ be the generator of the principal ideal $\fq^{h_k}$.  Since the class number of $k(\sqrt{\epsilon})$ is odd,  we have $\fQ^{h_k}=\alpha\co_{k(\sqrt{\epsilon})}$. Then 
 we produce a unit $\frac{\alpha^2}{\pi}$ of  $\co_{k(\sqrt{\epsilon})}$.   The unit group  $\co^\times_{k(\sqrt{\epsilon})}$ is isomorphic to $\Z^2\times \Z/{2\Z}$ by Dirichlet's unit theorem.
We claim that the following $3$-dimensional $\F_2$-vector space  \[\co^\times_{k(\sqrt{\epsilon})}/({\co^\times_{k(\sqrt{\epsilon})}})^2   \text{ is generated by }   -1, \sqrt{\epsilon}, \frac{\alpha^2}{\pi}. \]
Firstly we show that  $\frac{\alpha^2}{\pi}$ and $-1$ are linearly independent in this vector space. Suppose that $\frac{\alpha^2}{\pi}=\pm a^2$, then  $\sqrt{\pm \pi}\in k(\sqrt{\epsilon})$ and $k(\sqrt{\pm \pi})=k(\sqrt{\epsilon})$. It follows that $\pm \pi=\epsilon a^2$ with $a\in k$. This contradicts to that $\pi\co_k=\fq^{h_k}$, because $h_k$ is odd. 
It is easy to see that   $-1$  and $\sqrt{\epsilon}$ are linearly independent.  Suppose that these three elements are not linearly independent. Then  $\frac{\alpha^2}{\pi}=\pm \sqrt{\epsilon} b^2$ for some $b\in k(\sqrt{\epsilon})$. By taking norm from $k(\sqrt{\epsilon})$ to $\Q$, we obtain $-1=c^2$ for some $c\in \Q$. Here we use the fact $\epsilon$ has norm $-1$.    This is  a contradiction.  This   proves the claim.

Let $\fp_1,\fp_2$ be the prime ideals above $p$ in $k(\sqrt{\epsilon})$. They are  exactly the  places of $k(\sqrt{\epsilon})$ ramified in $L$. Note that $L=k(\sqrt{\epsilon})(\sqrt{\sqrt{p}})$. Then by Hasse's norm theorem, $\frac{\alpha^2}{\pi}\not\in N(L^\times) $ if and only if the quadratic Hilbert symbol  $(\frac{\alpha^2}{\pi},\sqrt{p})_{\fp_i}=-1$ for $i=1,2$.  Write $\pi=m+n\sqrt{p}$, then  $\pi \equiv m\bmod \sqrt{p}$ and $m^2-n^2p=2$.  Then
\[(\frac{\alpha^2}{\pi},\sqrt{p})_{\fp_i}=(\pi,\sqrt{p})_{\fp_i}=(m,\sqrt{p})_{\fp_i}=(m,p)_p=\leg{m}{p}=\leg{2}{p}_4.\]
 Hence \eqref{eq: 1mod16} is proved.  This finishes the proof of the theorem.\end{proof}

\begin{rmk}\label{rmk: 1mod16}
$(1)$ We remark that for $p\equiv 1\bmod 16$, then  $\leg{2}{p}_4=-1 \Longleftrightarrow 4\parallel h(-2p) \Longleftrightarrow 4\parallel h(-p)$, where $h(-p)$ is the class number of $\Q(\sqrt{-p})$.  See \cite{LW82}.

$(2)$ In order to compute the natural density of the set $\{p: p\equiv 1\bmod 16 \text{ and } \leg{2}{p}_4=-1 \}$, we consider the Galois extension $\Q(\zeta_{16},\sqrt[4]{2})/\Q$. Note that the degree of this extension is $16$. Let $\tau$ be the generator of $H:=\Gal(\Q(\zeta_{16},\sqrt[4]{2})/\Q(\zeta_{16}))$. Then the order of the conjugate class of $\tau$ is $1$ since $H$ is a  normal subgroup. 
A prime  $p\equiv 1\bmod 16$ and $\leg{2}{p}_4=-1$ if and only if the Frobenius of $p$ is $\tau$. By  Chebotarev's density theorem, the natural density of $\{p: p\equiv 1\bmod 16 \text{ and } \leg{2}{p}_4=-1 \}$ is $\frac{1}{16}$.

\end{rmk}

\section{Proof of Theorem~\ref{thm: 4_divides pure} and Theorem~\ref{thm: 4_hilbert class}.}
Let $p\equiv -1 \bmod 16$ be a prime in this section.
Put $L=K(\sqrt{2})$ and $F=\Q(\sqrt{p},\sqrt{2})$.  Let  $k=\Q(\sqrt{p}),k'=\Q(\sqrt{2p})$ and $k_0=\Q(\sqrt{2})$ be the  quadratic subfields of $F$.  In order to prove  Theorem \ref{thm: 4_divides pure}, we consider  the extensions $ F/k$, $L /F $ and $L/K$.
\begin{lem}\label{prop: equivalent of thm1 and thm2}
We have $4 \mid h_K$   if and only if $2\mid h_L$.
\end{lem}

\begin{proof}
Firstly note that $L/K$ is unramified outside $2$ because $\Q(\sqrt{2})/\Q$ is unramified outside $2$.  Since $p\equiv -1 \bmod 16$, $\Q_2(\sqrt[4]{p})=\Q_2(\sqrt[4]{-1})=\Q_2(\zeta_8)\supset \Q_2(\sqrt{2})$.  It follows that $L/K$ is unramified at every prime ideal above $2$, hence everywhere unramified.  By Hasse's norm theorem and local class field theory, we have  $[\co^\times_K:\co^\times_K\cap N(L^\times)]=1$.
Applying Chevalley's formula \eqref{eq: chevalley} to $L/K$ gives

\[|\Cl^{\Gal(L/K)}_L|=\frac{|\Cl_K|}{2}.\]
Therefore $4 $ divides  $|\Cl_K|$  $\Longleftrightarrow$   $ 2$ divides $|\Cl^{\Gal(L/K)}_L|$  $\Longleftrightarrow$ $ 2 $ divides $ |\Cl_L|$. 
\end{proof}

\begin{prop}\label{prop: units of base}
$(1)$ The class numbers  $h_{k},h_{k'},h_{k_0}$ and $h_F$ are all odd.

$(2)$ Let $\epsilon,\epsilon', 1+\sqrt{2}$ be the  fundamental units of $k,k',k_0$ respectively. Then $\co^\times_F=\langle \sqrt{\epsilon}, \sqrt{\epsilon'}, 1+\sqrt{2}\rangle \times \{\pm 1\}$.
\end{prop}

\begin{proof}
(1) The oddness of $h_{k},h_{k'},h_{k_0}$  is well-known.  Alternately, it is   easy to prove this  by applying Chevalley's formula.  We leave it to the readers.

Note that the only ramified prime in $F/k$ is the unique prime ideal of $k$ above $2$.
Take  $u \in \co^\times_{k}$,  by local class field theory $u$ is a local norm except at $2$.  However by the  Artin reciprocity law (or the product formula),  $u$ must be a   norm at $2$. Then  $u\in N_{F/k}(K^\times)$ by Hasse's norm theorem. Applying Chevalley's formula \eqref{eq: chevalley} to $F/k$  shows that $2\nmid h_F$.

(2) We first show that $\sqrt{\epsilon},\sqrt{\epsilon'} \in F$.  Let $\fq$ (resp. $\fq'$) be the unique prime ideal of $k$ (resp. $k'$) above $2$. Since both $h$ and $h'$ are odd and $2\co_k=\pi^2\co_k$ and $2\co_{k'}={\pi'}^2\co_k$. Then  $\frac{\pi^2}{2}\in \co^\times_{k}$ and $\frac{{\pi'}^2}{2}\in \co^\times_{k'}$. We may choose totally positive  generators  $\pi$ and $\pi'$   such that  $\epsilon=\frac{\pi^2}{2}$ and  $\epsilon'=\frac{{\pi'}^2}{2}$.  Thus  $\sqrt{\epsilon}=\frac{\pi}{\sqrt{2}}, \sqrt{\epsilon'}=\frac{\pi'}{\sqrt{2}}\in \co^\times_F$.  To prove the proposition, we only need to show that $[\co^\times_K:\co^\times_{k} \co^\times_{k'} \co^\times_{k_0}]=4$.

 Kuroda's class number formula   \cite[Theorem 1]{Lemmermeyer}  gives
\[[\co^\times_F:\co^\times_{k} \co^\times_{k'} \co^\times_{k_0}]= \frac{4h_F}{ h_{k}h_{k'}h_{k_0}}.\]
Since $h_K,h_{k},h_{k'},h_{k_0}$ are odd,  it suffices to  show that the units index is a power of $2$. Suppose not, let $r$ be an odd prime divides this index. Then there exists a unit $\eta \in \co^\times_F$ such that $\eta^r=\pm \epsilon^a{\epsilon'}^b(1+\sqrt{2})^c$ and $r\nmid \gcd(a,b,c)$. Note that  $N_{F/{k}}(\eta^r)=\pm \epsilon^{2a}$, this implies $r\mid a$. Similarly, $r\mid b$ and $r\mid c$. This  contradiction shows the index is a power of $2$, as desired. \end{proof}

\begin{lem}
$(1)$ The product of ramification indices for $L/F$ is $16$.

$(2)$ The index $[\co^\times_F:N(L^\times) \cap \co^\times_F]$ is $4$.
\end{lem}

\begin{proof}

(1) Obviously  $L/F$ is unramified outside the places above $2,p$ and the infinite places. In fact    $L/F$ is   unramified at the prime ideals  above $2$, because  \[\Q_2(\sqrt[4]{p},\sqrt{2})=\Q_2(\sqrt{p},\sqrt{2})=\Q_2(\zeta_8).\]
On the other hand, the  places  $\fp_1,\fp_2, \infty_1$  and $\infty_2$ are clearly ramified,  where $\fp_1,\fp_2$ are the two  prime ideals of $F$ above $p$ and  $\infty_1,\infty_2$  are the two real embeddings such that $\infty_i(\sqrt{p})<0$ and $\infty_i(\sqrt{2})=(-1)^i\sqrt{2}$ for $i=1,2$. This proves (1).

(2) As in the  proof of the above proposition. We let $\fq$ (resp. $\fq'$)  be the unique prime ideal  of $k$ (resp. $k'$) above $2$. Let  $\pi$ (resp. $\pi'$) be the  totally positive generators of  $\fq$ (resp. $\fq'$) such that $\frac{\pi^2}{2}$  (resp. $\frac{{\pi'}^2}{2}$)  are the fundamental unit of $k$ (resp. $k'$).  Then by the above proposition, 
\[\co^\times_F=\left\langle \frac{\pi_1}{\sqrt{2}},\frac{\pi_2}{\sqrt{2}},1+\sqrt{2},-1\right\rangle.\]

Since  $-1, \pm(1+\sqrt{2})$ are negative at $\infty_1$ or $\infty_2$,  they are not norms at $\infty_1$ or $\infty_2$ and then   not norms of $L$.  This shows  the index  $[\co^\times_F:N(L^\times) \cap \co^\times_F]\geq 4$. 

Now we go to show that
\[\left \langle \frac{\sqrt{2}(1+\sqrt{2})}{\pi_1},\frac{\sqrt{2}(1+\sqrt{2})}{\pi_2},(1+\sqrt{2})^2 \right\rangle \subset N(L^\times).\]
 Note that the left side is  a subgroup of $\co^\times_F$ with index $4$, so this would imply the desired results. Because the above units are totally positive, so they are norms at $\infty_1$ and $\infty_2$. For $\fp_1$ and $\fp_2$,  note that the localization of $F$ at $\fp_i$ ($i=1,2$) is $\Q_p(\sqrt{p})$, thus the proposition follows from  the following  lemma and Hasse's norm theorem.
\end{proof}

\begin{lem}\label{units are local squares}

	$(1)$ The elements  $ 2\pm \sqrt{2} $ are  squares in $ \Q^\times_p$.
	
	$(2)$ The elements  $\pi$ and $\pi'$ are squares in $\Q_p(\sqrt{p})^\times$.

\end{lem}

\begin{proof}
	
(1)Note that $(2+\sqrt{2})(2-\sqrt{2})=2$ is a square in $\Q^\times_p$.   So $2+\sqrt{2}\in (\Q^\times_p)^2$ if and only if $2-\sqrt{2}\in (\Q^\times_p)^2$.   Since $p\equiv -1 \bmod 16$, 
 $p$ splits completely in $\Q(\zeta_{16}+\zeta^{-1}_{16})$.  This implies that  $\zeta_{16}+\zeta^{-1}_{16} \in \Q_p$. Note that $(\zeta_{16}+\zeta^{-1}_{16})^2=\zeta_8+\zeta^{-1}_8+2=2+\sqrt{2}$ or $2-\sqrt{2}$.
	
(2) Write $\pi=a+b\sqrt{p}$ with $a \in \Z_{\geq 1},b \in \Z$. Then $a^2-pb^2=2$ and  $2\nmid ab$. By the quadratic reciprocity law for Jacobi symbols,
\[\leg{a}{p}=\leg{-p}{a}=\leg{2}{a}.\]
Note that $a^2\equiv 2 \bmod b$, in particular $\leg{2}{b}=1$. Hence  $b\equiv \pm 1\bmod 8$ and then  $b^2\equiv 1 \bmod 16$. Thus $a^2 = 2+pb^2 \equiv 1 \bmod 16$ and $\leg{2}{a}=1$. This implies that $\leg{a}{p}=1$. Hence  $\pi_1 \bmod \sqrt{p}$ is a square in $\Z_p[\sqrt{p}]/(\sqrt{p})$. By Hensel's lemma, $\pi_1$ is a square  in the local field $\Q_p(\sqrt{p}).$

Note that $\Q_p(\sqrt{p})=\Q_p(\sqrt{2p})$.
 Write $\pi'=c+d\sqrt{2p}$ with $c\in \Z_{\geq 1}, d\in \Z$. By Hensel's lemma, it is enough to prove that $\pi'$ is a square modulo $\sqrt{2p}$, or equivalently $c$ is a square modulo $p$.  Write $c=2^w c'$ with $2\nmid c'$. From the identity $c^2-2pd^2=2$, one has
\[\leg{c}{p}=\leg{2^wc'}{p}=\leg{c'}{p}=\leg{-p}{c'}=\leg{1}{c'}=1.\]   \end{proof}

We are now ready to prove the Theorem~\ref{thm: 4_divides pure}.
\begin{proof}[Proof of the Theorem~\ref{thm: 4_divides pure}]
 The Proposition \ref{prop: cyclic} proves the cyclic property.   The above two propositions and  Chevalley's formula \eqref{eq: chevalley} show that
	$|\Cl^{\Gal(L/K)}_L| =2.$ In particular, $2\mid h_L$.
	By Proposition \ref{prop: equivalent of thm1 and thm2},  we have $4\mid h_K$. This completes the proof.
\end{proof}

Now we go to prove Theorem~\ref{thm: 4_hilbert class}.
By Lemma~\ref{prop: equivalent of thm1 and thm2}, we see that $L$ is the $2$-Hilbert class field of $K$. To construct the $4$-Hilbert class field, we need the following results on the field  $K'=\Q(\sqrt{2\sqrt{p}})$.

\begin{prop} \label{prop: units of K'}The following statements are true:
	
$(1)$ The unique prime ideal $\pi\co_k$  above $2$ in $k$  splits in $K'/k$.
	
$(2)$ The class number of $K'$ is odd.	

$(3)$ There exists a totally positive unit $\xi$ of $\co^\times_{K'}$ such that $\co^\times_{K'}=\langle \xi,\bar{\xi},-1\rangle$ and $\N_{K'/k}(\xi)= \xi \bar{\xi}=\epsilon$.  
\end{prop}

\begin{proof}
For (1), note that  $\Q_2(\sqrt{p})=\Q_2(\sqrt{-1})$ and $\sqrt[4]{-p}\in \Q_2$. Then $\pi\co_k$ splits in $K'$ follows from $x^2-2\sqrt{p}=x^2-(1+\sqrt{-1})^2\sqrt{-p}=(x-(1+\sqrt{-1})\sqrt[4]{-p}))(x+(1+\sqrt{-1}\sqrt[4]{-p})) \text{ in } \Q_2(\sqrt{p})[x].$

For (2), note that the   places of $k$  ramified in $K'$ are $\sqrt{p}\co_k$ and the infinite place which sends $\sqrt{p}$ to $-\sqrt{p}\in \R$. Then   $-1\not\in \N(K'^\times)$. Since $\epsilon=\frac{\pi^2}{2}$ is totally positive, it is a norm at the infinite places. By product formula, $\epsilon$ must also be a norm at $\sqrt{p}\co_k$.  Then $\epsilon\in N(K'^\times)$.   By Chevalley's formula,  the class number of $K'$ is odd.

For (3), by \cite[Proposition 1.3.4]{Gre} and (2), we have $\epsilon \in N(\co^\times_{K'})$. Note that $\co^\times_{K'}/{\co^\times_k}$ is an abelian group of rank $1$. We claim that it is torsion-free. Otherwise, there exists $u\in \co^\times_{K'}\setminus \co^\times_k$ such that $u^j\in \co^\times_{K'}$ for some $j\geq 2$.  Then $K'=k(u)$.   The conjugate element of $u$ is $\zeta u$ for some $\zeta\in \langle \zeta_j \rangle \cap K'$ where $\zeta_j$ is a $j$-th primitive root of unity.  So $\zeta=\pm 1$ and   $N(u)=u \zeta u=\pm u^2\in \co^\times_k$.  This implies that  $K'/k$ is unramified at $p$. This is a contradiction. Hence we prove  the claim. 

Take $\eta \in \co^\times_{K'}$ such that it is a generator in $\co^\times_{K'}/{\co^\times_k}.$ We then have $\co^\times_{K'}=\langle -1,\eta,\epsilon\rangle.$  Since $\epsilon \in N(\co^\times_{K'})$,  the norm of $\eta$ must be an odd power of $\epsilon$, say $N(\eta)=\epsilon^{2k+1}$. Put  $\xi=\mathrm{sgn}(\eta)\eta\epsilon^{-k}$. Then $\xi$ is totally positive and  $N(\xi)=\epsilon$. Hence
 $\co^\times_{K'}=\langle -1,\xi,\epsilon\rangle= \langle -1,\xi,\bar{\xi}\rangle$.  
\end{proof}

\begin{proof}[Proof of Theorem~\ref{thm: 4_hilbert class}]
Let $M$ be the $4$-Hilbert class field of $K$. 
Since  $K/k$ is Galois,  $M/k$ is also Galois. By class field theory, one has
\[\Gal(M/K)\cong \Cl_K/{4 \Cl_K}  \quad \text{  as } \Gal(K/k) \text{-modules.}\]
 Here $\Gal(K/k):=\{1,\sigma\}$ acts on $\Gal(M/K)$ by conjugation.  Let $A_K$ be the $2$-primary part of $\Cl_K$.  Given an ideal class $a\in A_K$, we have $\N_{K/k}(a)=a \sigma(a)=1$  as $h_{k}$ is odd.  In other words, $\sigma(a)=a^{-1}$.  Hence $\Gal(M/k)$ is isomorphic to $D_4$, the dihedral group of order $8$. We then have a diagram of fields by Galois theory.
 
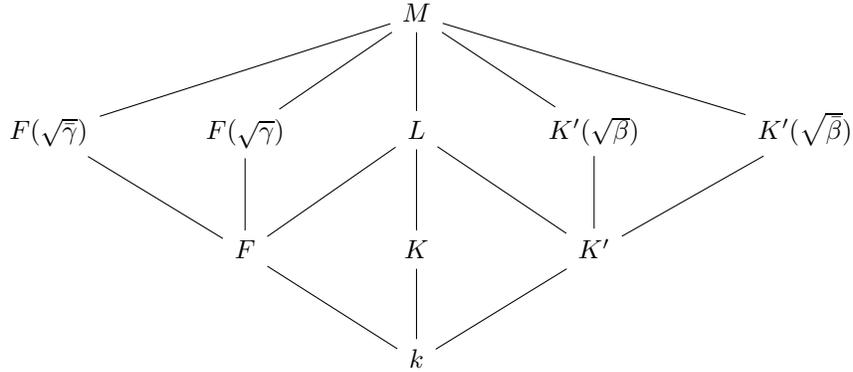
\[\begin{tikzcd}[row sep=large, column sep=large]
& & M\ar[lld,dash] \ar[ld,dash]  \ar[d,dash] \ar[rd,dash] \ar[rrd,dash]& & \\
F(\sqrt{\bar{\gamma}})&F(\sqrt{\gamma}) & L& K'(\sqrt{\beta})& K'(\sqrt{\bar{\beta}})\\
& F\ar[lu,dash] \ar[u,dash] \ar[ru,dash]& K \ar[u,dash]& K' \ar[lu,dash] \ar[u,dash] \ar[ru,dash]&\\
& & k\ar[lu,dash] \ar[u,dash] \ar[ru,dash] & & \\
 \end{tikzcd} \]
 Here $\beta$ is an element of ${K'}$ and $\bar{\beta}$ is the conjugate element of $\beta$ respect to the extension $K'/k$. Similarly,  $\gamma$ is an element of ${K}$ and $\bar{\gamma}$ is the conjugate element of $\gamma$ respect to the extension $K/k$. Note that $L=K'(\sqrt{\beta \bar{\beta}})=k(\sqrt{2\sqrt{p}}, \sqrt{\beta\bar{\beta}})$ and $k(\sqrt{\beta \bar{\beta}})=F$ or $K$.

  Since the class number $h'$ of $K'$ is odd,  we have $K'(\sqrt{\beta^{h'}})=K'(\sqrt{\beta})$. Let $\beta\co_{K'}=\fp^{a_1}_1\cdots \fp^{a_k}_k$ be the prime ideals factorization. Then $\beta':=\beta^{h'}=u\varpi^{a_1}_1\cdots \varpi^{a_k}_k$ where $u\in \co^\times_{K'}$ and $ \varpi_i\in \co_{K'}$ is a generator of $\fp^{h'}_i$ for each $i$.  We may assume  $a_i=0$ or $1$ for each $i$.  Since $F/k$ is unramified outside $2$,  so are $L/K'$ and $M/K'$. In particular,  $K'(\sqrt{\beta})/K'$ is unramified outside the prime ideals above $2$. This implies  $\beta'\co_{K'}= (\eta_1), (\eta_2), (\eta_1 \eta_2)$ or $\co_{K'}$.  Here $\eta_i \co_{K'}={\fq_i}^{h'}$ ($i=1,2$) and $\fq_i$ is the prime ideal above $\pi\co_{k}$. (Recall that $\pi\co_k$ is the unique prime above $2$ in $k$.)  Note that in each case  $k(\sqrt{\beta\bar{\beta}})/k$ is unramified at $\sqrt{p}\co_{k}$. So one must have  $k(\sqrt{\beta\bar{\beta}})=F$.

  We  claim  that $\beta'\co_{K'}=\co_{K'}$.    If $\beta'=u\eta_i$  for some unit $u$ ($i=1$ or $2$), then $\beta'\bar{\beta'}= v {\pi}^{h'}$. Because $F=k(\sqrt{\beta\bar{\beta}})=k(\sqrt{2})$, 
 $ v {\pi}^{h'}=2 t^2$  for some $t\in k^\times$ and  $v\in \co^\times_{k}$. This is a contradiction since $h'$ is odd. Hence $\beta'\neq u\eta_i$.  If $\beta'=u\eta_1\eta_2$, then both $\fq_1$ and $\fq_2$ are ramified in  the extensions  $K'(\sqrt{\beta'})/K'$, $K'(\sqrt{\bar{\beta'}})/K'$ and $ L/K'$  respectively. It follows then  $M/L$ is ramified at the primes above $2$. This is a contradiction since $M/L$ is unramified everywhere.  Therefore, we must have $\beta'$ is a unit in $\co_{K'}$.  This proves the claim.

  Since $K'(\sqrt{\beta})/K'$ is unramified at the infinite places, we have $K'(\sqrt{\beta})=K'(\sqrt{\xi}),  K'(\sqrt{\bar{\xi}})$ or $K'(\sqrt{\xi\bar{\xi}})$. 
 As $\xi\bar{\xi}=\epsilon\in k$,  this case can not happen. So $K'(\sqrt{\beta})=K'(\sqrt{\xi})$ or $K'(\sqrt{\bar{\xi}})$. It follows then $M=L(\sqrt{\xi})=LK'(\sqrt{\xi})=K(\sqrt{\xi}).$  This completes the proof of Theorem~\ref{thm: 4_hilbert class}.\end{proof}

\section{Distribution of the primes with $4\parallel h_K$}

In Section 4.1, we reprove the Theorem of Leonard-Williams \cite{LW82} on the  $16$-divisibility of the class number  $h(-2p)$ of $\Q(\sqrt{-2p})$.  In \cite{LW82}, they give a sketch of this proof by using the language of quadratic forms.  
  Milovic  says (see \cite[Page 976]{Milovic}) that he was  unable to verify the proof in \cite{LW82} and \cite[Proposition 1]{Milovic} gives a technical  proof by using $4$-Hilbert class field of $\Q(\sqrt{-2p})$.  We will follow the ideas in \cite{LW82} to prove this result in the language of ideals.  In Section 4.2, we prove Theorem~\ref{thm:density} and give some corollaries.

\subsection{$16$-divisibility  of $h(-2p)$}

\begin{lem} \label{(p) is well-defined}
  Let $p\equiv -1\bmod 8$ be a prime. For any decomposition $p=u^2-2v^2$ with $u,v \in \mathbb{N}$, the Jacobi symbol $\leg{2u}{v}$ is independent on the choices of $u$ and $v$.
\end{lem}

\begin{proof}
  Suppose $p={u'}^2-2{v'}^2$ with $u',v'\in \mathbb{N}$. Then $u'+v'\sqrt{2}=(1+\sqrt{2})^{2k}(u+v\sqrt{2})$ or $(1+\sqrt{2})^{2k}(u-v\sqrt{2})$ for some $k\in \Z$.
  It suffices to consider  $k=1$.  In  the first case,  $u'=3u+4v$ and $v'=2u+3v$. Since $p\equiv -1\bmod 8$,  we have  that $u,v$ are odd and  that  $v\equiv 2u+3v\bmod 4$.
  Then $\leg{2u'}{v'}=\leg{-v}{2u+3v}=\leg{2u}{v}$ by the  quadratic reciprocity law. The second case can be proved similarly.
\end{proof}

Recall that we denote this Jacobi symbol $\leg{2u}{v}$ by $(p)$. 

\begin{thm}\label{thm: 16-rank formula}
  Let $p\equiv -1\bmod 8$ be a prime. Then the following statements hold:

  $(1)$  (Hasse)  $4\parallel h(-2p) \Longleftrightarrow p \equiv 7 \bmod16.$

  $(2)$  (Leonard-Williams)   $8\parallel h(-2p) \Longleftrightarrow  p\equiv -1 \bmod 16  \text{ and }(-1)^{\frac{p+1}{16}}(p)=-1.$
\end{thm}

\begin{proof}
Let $J_2$ be the  prime ideal  of $k:=\Q(\sqrt{-2p})$ above 2. Then we have   that  $N(J_2)=|\co_k/{J_2}|=2$ and that  $J_2$ has order $2$ in the class group $\Cl_k$.
One can write  $p=u^2-2v^2$ such that $u\in \mathbb{N}$ and $u\equiv1\bmod 4$. Hence $(2v+\sqrt{-2p})(2v-\sqrt{-2p})=2u^2$. Note that the ideals $ (2v+\sqrt{-2p}) J^{-1}_2$ and $ (2v-\sqrt{-2p}) J^{-1}_2$  are coprime.  It follows that  there exists an ideal $J_u$  such that $(2v+\sqrt{-2p}) =J^2_uJ_2$. Automatically,  $N(J_u)=u$  and  $\cl(J_u)$ has order $4$ in $\Cl_k$.

Since the genus field of $k$ is $k(\sqrt{2})$. By class field theory, the Artin map induces an isomorphism
\[ \big[\frac{k(\sqrt 2) / k}{-}\big]: \Cl_k/{2\Cl_k}\cong \Gal(k(\sqrt{2})/k).\]
Therefore, we have
\[\begin{split}
8|h(-2p) \Longleftrightarrow  & \cl(J_u)\in 2\Cl_k \\
\Longleftrightarrow  &\big[\frac{k(\sqrt 2) / k}{J_u}\big]=1 \Longleftrightarrow  \big[\frac{\Q(\sqrt 2) / \Q}{N(J_u)}\big]=\leg{2}{u}=1  \text{ (by norm-functoriality of Artin map)}\\
\Longleftrightarrow & u^2\equiv 1\bmod 16  \text{  } (\Longleftrightarrow  u\equiv 1\bmod 8, \text{ since } u\equiv 1\bmod 4)\\
\Longleftrightarrow & p\equiv -1\bmod 16 \text{ (since } p=u^2-2v^2 \text{ with } u,v \text{ odd}). \\ \end{split}\]

Now let us  prove (2).    Since $8\mid h(-2p)$,  we  choose an integral ideal $J_s$ with norm $s$ such that $\cl(J_s)^2=\cl(J_u)$. By approximation theorem,   we can further assume that $\gcd(s,2up)=1$.  Hence $\cl(J_s)$ has order $8$ in $\Cl_k$ and one has that 
\[16\mid h(-2p) \Longleftrightarrow \cl(J_s) \in 2\Cl_k  \Longleftrightarrow \big[\frac{k(\sqrt 2) / k}{J_s}\big]=1 \Longleftrightarrow  \leg{2}{s}=1.\]

Thus it remains to prove that
\begin{equation}\label{eq:  main in pf of 16rank}
\leg{2}{s}=(-1)^{\frac{p+1}{16}}(p).  \tag{$\ast \ast$}
\end{equation}

   In the decomposition $p=u^2-2v^2$,  we have $u\equiv 1\bmod 8$ by (1).
   We may  further  assume that  $u\equiv 1\bmod 16$ by replacing   $u+\sqrt{2}v$ to   $(1+\sqrt{2})^4(u+v\sqrt{2})$.
Let  $x+y\sqrt{-2p}$  and $z+w\sqrt{-2p}$  denote the generators of the  principal ideals $\overline{J_u}J_s^2$ and  $J_2J_uJ_s^2$, respectively. Here  $\overline{J_u}=(u)J^{-1}_u$ is the conjugate ideal of $J_u$ and we assume $x>0$.
By taking norm of the equality $\overline{J_u}J_s^2=(x+y\sqrt{-2p})$, we have
\begin{equation}\label{eq:  1 pf of 16rank}
us^2=x^2+2py^2.
\end{equation}

From the equality of ideals
\[(z+w\sqrt{-2p})=J_2J_uJ^2_s=J_2J^2_u J^{-1}_u J^2_s=(2v+\sqrt{-2p})u^{-1}(x+y\sqrt{-2p}),\]
we can change the sign of $z, w$ such that   \begin{equation} \label{eq:  2 pf of 16rank}
uw=x+2vy.
\end{equation}

From \eqref{eq:  1 pf of 16rank}, \eqref{eq:  2 pf of 16rank} and $p=u^2-2v^2$, we obtain
\begin{equation}\label{eq:  3 pf of 16rank}
  s^2=uw^2-4vyw+2uy^2.
\end{equation}

We claim that $s^2\equiv w^2\bmod 16$.  Since $s$ is odd,  we know that  $y$ is even and that $x$ is odd from \eqref{eq:  1 pf of 16rank}.
From \eqref{eq:  2 pf of 16rank}, we have $w$ is odd. Then the claim holds by \eqref{eq:  3 pf of 16rank}.

Therefore  \begin{equation}\label{eq: 4 pf of 16rank}
\leg{2}{s}= \leg{2}{|w|}=\leg{u}{|w|}=\leg{w}{u}=\leg{v}{u} \leg{y}{u}=\leg{v}{u}=\leg{u}{v}.
\end{equation}
The first equality is by the above claim.
The second and the fourth equalities are by \eqref{eq:  3 pf of 16rank}. The third and the last equalities are by the quadratic reciprocity law and the fact that  $u\equiv 1\bmod 16$.   To see the fifth equality, write $y=2^t y_0$ with $2\nmid y_{0}$. Then by $u\equiv 1\bmod 16$,  we have  $\leg{y}{u}=\leg{2^t y_{0}}{u}=\leg{ y_{0}}{u}=\leg{u}{|y_{0}|}=1$. The last equality   follows from \eqref{eq:  1 pf of 16rank}.

From $p=u^2-2v^2$, we obtain that $2v^2\equiv 1-p\bmod 32$ as $u\equiv 1\bmod 16$. Then
\[ \leg{2}{v}=(-1)^{\frac{p+1}{16}}.  \]
Thus \eqref{eq: 4 pf of 16rank} and the above equality imply \eqref{eq:  main in pf of 16rank}.  This proves the proposition.  \end{proof}

\subsection{Proof of Theorem~\ref{thm:density}}

 The following is  the main result  in \cite[Theorem~2]{Milovic}.
 \begin{equation} \label{eq: milovic}
 \sum\limits_{p \leq X \atop p\equiv -1\bmod16} (-1)^\frac{p+1}{16} (p) \ll_\epsilon X^{\frac{149}{150}+\epsilon}.\end{equation}
  Here $f(X)\ll_\epsilon g(X)$ means that for each $\epsilon>0$, there exists some positive constant $C_\epsilon$ such that $|f(X)|\leq C_\epsilon g(X)$.  Then by  Theorem~\ref{thm: 16-rank formula}, Milovic obtains the density of the set of primes  $p$ such that $p\equiv 3\bmod 4$ and $16\mid h(-2p)$. 
  
We will show  that  Milovic's method can be used to prove the following.
  \begin{equation}\label{eq: main}
  A^{+}(X)- A^{-}(X)=\sum\limits_{p \leq X \atop p\equiv -1\bmod16} (p) \ll_\epsilon X^{\frac{149}{150}+\epsilon} .
  \end{equation}
Here $A^{\pm}(X)=\#\{p \leq X, p\equiv -1\bmod16, (p)=\pm1\}$. On the other hand,  by Dirichlet's density theorem, 
\begin{equation*}
  A^{+}(X)+ A^{-}(X)= \sum\limits_{p \leq X \atop p\equiv -1\bmod16} 1 \sim  \frac{1}{8}\frac{X}{\log X}.
\end{equation*}
Here $f(X) \sim g(X)$ means the limit of $f(X)/g(X)$ is $1$ as $X\rightarrow \infty$. We conclude that 
\[A^{-}(X)\sim \frac{1}{16}\frac{X}{\log X}.\] 
 Theorem~\ref{thm:density} then follows.  It remains to prove  \eqref{eq: main}. 
 
 To prove \eqref{eq: milovic},   Milovic    \cite[Section 3.2]{Milovic} defines the spin symbol for all totally positive elements of $\Z[\sqrt{2}]$ as follows.
 \[   [u+v\sqrt{2}]= \begin{cases}
\leg{v}{u}, \text{ if }  u \text{ is odd }\\
0,  \quad \text{ otherwise. }
 \end{cases} \]

In order to prove our result, we define the twisted spin symbol for all totally positive elements of $\Z[\sqrt{2}]$ as 
\[[u+v\sqrt{2}]'=[u+v\sqrt{2}] \lambda(u+v\sqrt{2}),\] where $\lambda(u+\sqrt{2}v)=(-1)^{\frac{u^2-2v^2+1}{16}}$ if $u^2-2v^2\equiv -1\bmod16$ and 1 otherwise.  
One can easily show that the twisted  spin symbol satisfies
$[(1+\sqrt{2})^8(u+v\sqrt{2})]'=[u+v\sqrt{2}]'$ which is an analogous property of the spin symbol $[u+v\sqrt{2}]$ as in \cite[Proposition 2]{Milovic}.  Then one can replace $[u+v\sqrt{2}]$ in \cite{Milovic} by $[u+v\sqrt{2}]'$ to follow   Milovic's argument \cite[Page 993-1013]{Milovic}.  This  gives a proof of \eqref{eq: main} and then Theorem~\ref{thm:density}.

Using \eqref{eq: milovic} and \eqref{eq: main}, we  obtain   refined results on $h(-2p)$ for $p\equiv -1\bmod16$. 
\begin{cor}We have
  \[ \begin{split}
&\lim\limits_{X \rightarrow \infty} \frac{\#\{p \leq X: p \equiv 15 \bmod32 \text{ and }8\parallel h(-2p)\}}{\#\{p \leq X: p \equiv 15 \bmod32 \}}  \\ = 
&\lim\limits_{X \rightarrow \infty} \frac{\#\{p \leq X:  p \equiv 31 \bmod32 \text{ and } 8\parallel h(-2p)\}}{\#\{p \leq X: p \equiv 31 \bmod32 \}}=\frac{1}{2}.
  \end{split}\]
\end{cor}

\begin{proof}
  By \eqref{eq: milovic} and \eqref{eq: main}, we have
  \[2\sum_{p \leq X \atop p\equiv -1\bmod32}(p)=\sum_{p \leq X \atop p\equiv -1\bmod16} (-1)^\frac{p+1}{16} (p)+(p) \ll_\epsilon X^{\frac{149}{150}+\epsilon}  .\]
 By Theorem~\ref{thm: 16-rank formula}, for  $p\equiv -1\bmod32$ one has  $8\parallel h(-2p) $ if and only if  $ (p)=-1$. Hence the result follows the Dirichlet's density theorem.
  The case  $p\equiv 15\bmod32$  can be shown similarly.
\end{proof}


\begin{thebibliography}{3}

\bibitem{Hasse}
 H. Hasse. \"{U}ber die Klassenzahl des K\"{o}rpers $P(\sqrt{-2p})$ mit einer Primzahl $p\neq 2$.
J. Number Theory. 1 (1969), 231-234.

\bibitem{Lemmermeyer}F. Lemmermeyer.  Kuroda's class number formula. Acta Arith. 66 (1994), 245-260.

\bibitem{Lemmermeyer2}F. Lemmermeyer. The ambiguous class number formula revisited. J. Ramanujan Math. Soc. 28 (2013), 415-421.

\bibitem{ljn} J. Li, Y. Ouyang,  Y. Xu and  S. Zhang. $\ell$-Class groups of fields in  Kummer towers. \url{https://arxiv.org/abs/1905.04966}.

\bibitem{Gre}
R. Greenberg. Topics in Iwasawa Theory. Website: \url{https://sites.math.washington.edu/~greenber/book.pdf}.
	
\bibitem{Monsky} P. Monsky. A result of Lemmermeyer on class numbers.  \url{https://arxiv.org/abs/1009.3990}.

\bibitem{Parry}
C.J. Parry. A genus theory for quartic fields. J. Reine Angew. Math. 314 (1980),  40-71.

\bibitem{LW82}
P.A. Leonard and K.S. Williams. On the divisibility of the class numbers of $\Q(\sqrt{-p})$ and $\Q(\sqrt{-2p})$ by $16$. Canad. Math. Bull
25(2)  (1982), 200-206.

\bibitem{Milovic}
D. Milovic. On the 16-rank of class groups of  $\Q(\sqrt{-8p})$ for $p \equiv -1 \bmod 4$. Geom. Funct. Anal. 27 (2017),  973-1016.


\bibitem{Redei}
L. R\'{e}dei. Arithmetischer Beweis des Satzes \"{u}ber die Anzahl der durch vier teilbaren Invarianten der absoluten Klassengruppe im quadratischen Zahlk\"orper. J. Reine Angew. Math. 171 (1934), 55–60.

\bibitem{Reichardt}
H. Reichardt. Zur Struktur der absoluten Idealklassengruppe im quadratischen Zahlk$\ddot{o}$rper. J. Reine Angew. Math. 170 (1934), 75–82.

\bibitem{Pari}
The PARI Group, Univ. Bordeaux, PARI/GP version 2.7.5, 2015, available from \url{http://pari.math.u-bordeaux.fr/.}


\end{thebibliography}
\end{document}